\documentclass{amsart}

\usepackage[colorlinks,citecolor=niceblue,linkcolor=niceblue]{hyperref}
\usepackage{enumitem}
\usepackage{esint}
\usepackage{amssymb}
\usepackage{tikz}
\usepackage{graphicx}
\usepackage{amsrefs,enumitem,pgfkeys}

\definecolor{niceblue}{rgb}{0,0,0.7}

\newtheorem{theorem}{Theorem}[section]
\newtheorem{lemma}[theorem]{Lemma}
\newtheorem{corollary}[theorem]{Corollary}

\theoremstyle{definition}

\newtheorem{remark}[theorem]{Remark}

\newtheorem{definition}[theorem]{Definition}

\newcommand{\eref}[1]{(\ref{e.#1})}
\newcommand{\tref}[1]{Theorem \ref{t.#1}}
\newcommand{\lref}[1]{Lemma \ref{l.#1}}

\newcommand{\cref}[1]{Corollary \ref{c.#1}}

\newcommand{\sref}[1]{Section \ref{s.#1}}

\newcommand{\dref}[1]{Definition \ref{d.#1}}

\numberwithin{equation}{section}

\newcommand{\R}{\mathbb{R}}

\newcommand{\grad}{\nabla}

\def\XXint#1#2#3{{\setbox0=\hbox{$#1{#2#3}{\int}$ }
\vcenter{\hbox{$#2#3$ }}\kern-.6\wd0}}

\definecolor{darkgreen}{rgb}{0,0.4,0}

\begin{document}

\title[Convex comparison for anisotropic Bernoulli problems]{On convex comparison for exterior Bernoulli problems with discontinuous anisotropy}
\keywords{Bernoulli problem, comparison principle, discontinuous anisotropy}
\subjclass{35R35, 35E10}
\author[W. M. Feldman]{William M Feldman}
\address{Department of Mathematics, University of Utah, Salt Lake City, Utah, USA}
\email{feldman@math.utah.edu}
\author[N. Po\v{z}\'{a}r]{Norbert Po\v{z}\'{a}r}
\address{Faculty of Mathematics and Physics, Institute of Science and Engineering, Kanazawa University, Kakuma, Kanazawa 920-1192, Japan}
\email{npozar@se.kanazawa-u.ac.jp}
\maketitle
\begin{abstract}
    We give a new proof of a convex comparison principle for exterior Bernoulli free boundary problems with discontinuous anisotropy.
\end{abstract}
\section{Introduction}

Let $K$ be a closed convex subset of $\R^d$ with nonempty interior.  We will consider the convexity properties of the minimal supersolution of the anisotropic Bernoulli problem
\begin{equation}\label{e.minsup}
\left\{
\begin{aligned} 
 \Delta u &= 0 && \hbox{in } \ \{u>0\} \setminus K, \\
 |\grad u| &= Q(n_x) && \hbox{on } \ \partial \{u>0\}, \\
 u &= 1 && \hbox{on } K.
 \end{aligned}  
\right.
 \end{equation}
Here $n_x$ is the inner unit normal to $\partial\{u > 0\}$ at $x$ and $Q: S^{d-1} \to (0, \infty)$ is only assumed to be bounded above and below and upper semi-continuous 
\[ Q(n) \geq \limsup_{n' \to n} Q(n').\]
Such problems, with discontinuous anisotropy, naturally arise from periodic homogenization scaling limits \cite{FeldmanSmart,Feldman,CaffarelliLee,Kim,CaffarelliMellet}.  

In the context of homogenization and other scaling limits, one cannot directly show the minimality property of the limiting supersolutions.  Instead the following weak subsolution property is natural.  Essentially this is a viscosity subsolution property that only allows to test the free boundary condition with one-dimensional test functions.

\begin{definition}
\label{d.one-dimensional}
We say that $\varphi \in C^\infty(U)$ is \emph{one-dimensional in $U$} if it is of the form $\varphi(x) = f(x\cdot p)$ in $U$ for some $f\in C^\infty(\R)$ and $p \in \R^d$, $|p| = 1$.
\end{definition}

\begin{definition}\label{d.weaksub}
A \emph{weak subsolution} of \eref{minsup} is a nonnegative function $u \in C(\R^d)$ that is compactly supported, satisfies $u \leq 1$ on $K$, is harmonic in $\{u > 0\} \setminus K$, and such that, whenever $U$ is an open neighborhood and one-dimensional $\varphi$ with $\Delta \varphi < 0$ touches $u$ from above in $\overline{\{u > 0\}}$ at $x \in \partial\{ u > 0 \} \cap U$ with strict ordering $u < \varphi$ on $\overline{\{u > 0\}} \cap \partial U$, then
\[|\nabla \varphi(x)| \geq Q(p),\]
where $p$ is from \dref{one-dimensional}.
\end{definition}

The following supersolution definition is standard.

\begin{definition}
\label{d.supersol}
A supersolution of \eqref{e.minsup} is a nonnegative function $u\in C(\R^d)$ that is compactly supported, $u \geq 1$ on $K$, is harmonic in $\{u > 0\} \setminus K$, and whenever $\varphi \in C^\infty(U)$, $U$ open, $\Delta \varphi > 0$ and $\nabla \varphi \neq 0$ in $U$, touches $u$ from below at $x \in \partial \{u > 0\} \cap U$ then
\begin{align*}
|\nabla \varphi(x)| \leq Q\left(\frac{\grad \varphi}{|\grad \varphi|}(x)\right).
\end{align*}
\end{definition}

We give a new proof of the following theorem.

\begin{theorem}
A supersolution of \eref{minsup} is minimal if and only if it satisfies the weak subsolution property \dref{weaksub}. In other words, a supersolution that is also a weak subsolution is the unique minimal supersolution. Furthermore the minimal supersolution has convex superlevel sets.
\end{theorem}
This result has already been proved in \cite{FeldmanSmart} by Smart and the first author.  This comparison principle, for discontinuous anisotropies, was an important component in the proof of the homogenization scaling limit of minimal supersolutions in that work and in \cite{Feldman}.  This paper gives a new proof which we feel has valuable simplicity and intuition. The main new ideas in this paper are in \lref{subsol-grad-bound} and \cref{cts-approx} below. We give a brief explanation about those new ideas here. 

The difficulty in using the weak subsolution condition is that, on a facet $F$ (see \sref{extreme-points} for definitions) with inward normal $n$ of the free boundary of a quasi-convex subsolution $u$, the weak subsolution condition only guarantees that 
\[ \max_{F} |\grad u| \geq Q(n).\]
Unfortunately if a supersolution touches $u$ from above on this facet it may not touch at the point where the maximum in the subsolution condition is saturated. We would like instead the strong subsolution condition
\begin{equation}\label{e.strong-subsolution}
    \min_{F} |\grad u| \geq Q(n).
\end{equation}
The first idea, exploited in \lref{subsol-grad-bound}, is that the weak and strong subsolution conditions above are the same at exposed points, facets $F$ which are singletons.  By Straszewicz's Theorem, \tref{Straszewicz} quoted below, all extreme points are limits of exposed points so, along with some additional convexity arguments we can show that weak subsolutions actually satisfy the strong condition \eref{strong-subsolution} on all facets at least when $Q$ is continuous.

This idea does not directly work in the case of upper semicontinuous $Q$; if it did, all facets of a solution would actually be trivial.  However it turns out, somewhat unexpectedly, that we can use a simple monotone continuous approximation argument of the upper semicontinuous $Q$ to prove a comparison principle in the discontinuous case too.  This is the content of \cref{cts-approx} below.

There have been several works on existence, uniqueness, convexity, and regularity of solutions of Bernoulli-type problem \eref{minsup} with $Q(n) \equiv 1$ or continuous. Classical work of Beurling \cite{Beurling}, and later Schaeffer \cite{Schaeffer} and Acker \cite{Acker}, showed the existence/uniqueness of a convex solution in $2$-d. Hamilton \cite{Hamilton} gave a different proof in $2$-d by a Nash-Moser implicit function theorem.  Later Henrot and Shahgholian \cite{HenrotShahgholian1997,HenrotShahgholian2002} gave proofs based on maximum principle including nonlinear problems with $x$-dependent boundary conditions (satisfying a concavity assumption). Bianchini \cite{Bianchini2012} extended this to anisotropic interior Bernoulli problems and to Finsler metrics \cite{Bianchini2022}. The present paper is most similar to these works \cite{HenrotShahgholian1997,HenrotShahgholian2002,Bianchini2012,Bianchini2022}; for us the main difficulty is in proving comparison principle with only the weak subsolution condition.

\subsection*{Acknowledgments}  The authors appreciate discussions with Inwon Kim which helped to motivate the proof idea. W. Feldman was partially supported by the NSF grant DMS-2009286. N. Po\v{z}\'ar was partially supported by JSPS KAKENHI Kiban~C Grant No. 23K03212.

\section{Convexity properties}
  \subsection{Extreme and exposed points of convex sets}\label{s.extreme-points} We recall several notions and results from convex analysis; see for example \cite[Sec.~18]{Rockafellar}.

For a convex set $X$, a \emph{face} is a convex subset $Y \subset X$ such that every line segment in $X$ with relative interior point\footnote{$(1-t) x + ty$ for $t \in (0,1)$ are relative interior points of segment $x$--$y$.} in $Y$ has both end points in $Y$. If a face of $X$ is a point, it is called an \emph{extreme point} of $X$.

The convex set $Y = \{x \in X : h(x) = \max_X h\}$ for some linear function $h$ is a face of $X$, and it is called an \emph{exposed face}.
If an exposed face $X$ is a point, it is called an \emph{exposed point} of $X$. Also we call a \emph{facet} of $X$ to be a nontrivial exposed face, i.e. not the entire set $X$.

The following two results about extreme points and exposed points will play a key role.

\begin{lemma}[{\cite[Cor.~18.5.1]{Rockafellar}}]
\label{l.extremehull}
A closed bounded convex set is the convex hull of its extreme points.
\end{lemma}

\begin{theorem}[Straszewicz's Theorem{\cite[Th.~18.6]{Rockafellar}}]
\label{t.Straszewicz}
Let $X$ be a closed convex set. Every extreme point of $X$ is the limit of a sequence of exposed points of $X$.
\end{theorem}

\subsection{Convexity properties of a gradient of a harmonic function}

Let $\emptyset \neq U \subset\subset V$ be two bounded open convex sets. We also assume that $V$ is inner regular. Let $v$ be the unique solution of
\begin{align}
\label{dirichlet}
\left\{\begin{aligned}
\Delta v &= 0 && \text{in } V \setminus \overline{U},\\
v &= 1 && \text{on } \partial U,\\
v &= 0 && \text{on } \partial V.
\end{aligned}\right.
\end{align}
By \cite[Th.~2.1]{CaffarelliSpruck} the super-level sets $\{v > t\}$ are convex for all $t \in (0, 1)$. This implies convexity of the super-level sets of $|\grad v|$ on the facets of $\partial V$.

\begin{lemma}
\label{l.grad-convex}
Let $U$, $V$ and $v$ be as above and let $\Lambda$ be a tangent hyperplane to $\partial V$. Recall that $V$ is uniformly inner regular. Set $F = \partial V \cap \Lambda$. Then $1 / |\grad v|$ is convex on $F$. In particular, its sublevel sets
\begin{align*}
\{x \in F: 1 / |\grad v(x)| < 1/\tau\} = \{x \in F: |\grad v(x)| > \tau\}
\end{align*}
are convex for any $\tau > 0$.
\end{lemma}

\begin{proof}
Let $n$ be the unit normal of $\Lambda$ so that it is also the inner unit normal of $\partial V$ on $F$.

Fix $\xi_0, \xi_1 \in F$ and set $\xi_t := (1 - t) \xi_0 + t \xi_1$. For any $\theta > 0$ we define
\begin{align*}
h_\theta(t) := \inf \{h > 0: v(\xi_t + h n) > \theta\}, \qquad t \in [0,1].
\end{align*}
By continuity of $v$ we have $h_\theta(t) > 0$. Since $\{v > \theta\} \nearrow V$ as $\theta \to 0$ and $\{v > 0\}$ is inner regular at $\xi_t$ we have $h_\theta(t)  < \infty$ for sufficiently small $\theta > 0$ for all $t \in [0, 1]$. In fact, $h_\theta(t) \to 0$ as $\theta \to 0$.
Since $\{v > \theta\}$ is convex by \cite{CaffarelliSpruck}, $h_\theta$ is convex on $[0,1]$. 

Therefore
\begin{align*}
|\grad v(\xi_t)| = \grad v(\xi_t) \cdot n = \lim_{\theta \to 0} \frac{v(\xi_t + h_\theta(t) n)}{h_\theta(t)} = \lim_{\theta \to 0} \frac{\theta}{h_\theta(t)}.
\end{align*}
By inner regularity and Hopf's lemma $|\grad v(\xi_t)| > 0$ for all $t \in [0,1]$ and hence
we deduce that $t \mapsto 1 / |\grad v(\xi_t)|$ is convex since $t \mapsto h_\theta(t)/\theta$ is convex and pointwise limits of convex functions are convex.
\end{proof}

\begin{remark}
It is not in general true that $|\grad v|$ is concave on $F$.
Here is a counter-example on an unbounded $V$.

In $d = 2$, consider the top half-space $V := \{x_2 > 0\}$ and the potential between a point charge at $(0, 1)$ and the plane $\{x_2 = 0\} = \partial V$:
\begin{align*}
v(x) = -\log(x_1^2 + (x_2 - 1)^2) + \log(x_1^2 + (x_2 + 1)^2),
\end{align*}
and set $U = \{v > 1\}$, a convex set.
We have
\begin{align*}
|\grad v(x)| = \frac{\partial v}{\partial x_2}(x_1, 0) = \frac 4{x_1^2 + 1}, \qquad x \in \partial V,
\end{align*}
which is not convex.
\end{remark}

\section{Comparison principle}

In this section we establish the proof of comparison between a weak subsolution and a supersolution when at least one has convex support. When the supersolution has convex support, we are able to directly show comparison with only semicontinuous $Q$. However, if only the weak subsolution has convex support we need to additionally assume that $Q$ is continuous. To finally use this result to establish uniqueness of function that is both a supersolution and a weak subsolution with convex support when $Q$ is upper semicontinuous, we monotonically approximate $Q$ by continuous functions.

We recall the notions of sup- and inf-convolutions that we will use to regularize free boundaries of the solutions. For a set $V$ we define
\begin{align}
\label{sub-inf-conv}
V^r := V + B_r, \qquad V_r := \{x : B_r(x) \subset V\}.
\end{align}
For a continuous $u$, we then define $u^r$ as the solution of \eqref{dirichlet} with $K^r$ and $\{u > 0\}^r$ in place of $\overline U$ and $V$. Similarly replacing $\overline U$ by $K_r$ and $V$ by $\{u > 0\}_r$ leads to $u_r$.

If $u$ is a weak subsolution so is $u^r$, and furthermore $\{u^r > 0\}$ is uniformly inner regular with radius $r$. Similarly, if $u$ is supersolution so is $u_r$ and $\{u_r > 0\}$ is uniformly outer regular. 

Let us recall a stability property of supersolutions.
\begin{lemma}\label{l.stability}
Let $Q$ be a upper semicontinuous function. Then supersolutions in the sense of \dref{supersol} are stable under uniform convergence.
\end{lemma}

\begin{proof}
Let $\{u_k\}$ be a sequence of supersolutions and let $u$ be its uniform limit. Clearly $u$ is continuous, $u \geq 1$ on $K$ and $u$ is harmonic in $\{u > 0\}$.

Suppose that $\varphi \in C^\infty(U)$, $\Delta \varphi > 0$ and $\nabla \varphi \neq 0$ in $U$, touches $u$ from below at $x \in \partial \{u > 0\} \cap U$. We can assume that $u - \varphi$ has a strict minimum at $x$. By uniform convergence, all points of minimum of $u_k - \varphi$ converge to $x$ as $k \to \infty$. By maximum principle for harmonic functions the minima are located on $\partial \{u_k > 0\} \cap U$ for sufficiently large $k$. Let $x_k$ be a sequence of such minima. For large $k$ we have $\nabla \varphi (x_k) \leq Q(\nabla \varphi(x_k)/ |\nabla \varphi(x_k)|)$. By sending $k \to \infty$ and by the upper-semicontinuity of $Q$ we deduce that this also holds at $x$. Therefore $u$ is a supersolution.
\end{proof}

\subsection{Outer regular points have a classical normal derivative} First we show a useful technical lemma: the normal derivative is well defined at outer regular free boundary points, and for a supersolution it satisfies the supersolution condition.

\begin{lemma}\label{l.supersoln-blowup}
    Suppose that $u$ is a supersolution, and $x_0 \in \partial \{u>0\}$ is an outer regular free boundary point.  Then $|\grad u(x_0)|$ is well-defined and
    \[ |\grad u(x_0)| \leq Q(n_{x_0}).\]
\end{lemma}
\begin{proof}
    Without loss assume that $x_0 = 0$ and outward normal determined by the exterior ball is $-e_d$. By \cite[Lemma 11.18]{CaffarelliSalsa} the blow-up sequence
    \[ u_r(x) = \frac{u(rx)}{r} \to \alpha x_d \ \hbox{ in non-tangential cones for some } \ \alpha \in [0,\|\grad u\|_\infty].\]
    On the other hand the blow-up sequence $u_r(x)$ is uniformly Lipschitz continuous with $u_r(0) = 0$ and so every subsequence has a subsequence converging uniformly on $\R^n$.  Every subsequential limit must be zero on $x_n \leq 0$ by the exterior ball condition, and must agree with $\alpha x_n$ on $x_n > 0$ by the non-tangential limit.  Thus the sequence actually converges locally uniformly to $\alpha (x_d)_+$.  
    %\will{By (relatively) standard uniform stability of viscosity solution properties}, 
    By \lref{stability} for upper semicontinuous $Q$, the limit $\alpha (x_n)_+$ is also a viscosity supersolution and so $\alpha \leq Q(e_d)$.
\end{proof}

\subsection{Comparison principle: convex supersolution}
\begin{lemma}\label{l.comparison-convex-above}
Assume $Q$ is upper semicontinuous. Suppose that $v$ is supersolution of \eref{minsup} with $V:=\{v>0\}$ convex, and $u$ is a weak subsolution of \eref{minsup}.  Then $v \geq u$.
\end{lemma}
\begin{proof}
By translation we can assume that the interior of $K$ contains the origin.
Since $v$ and $u$ are harmonic in their positive sets, it is sufficient to show that $ \{u>0\} =: U \subset V := \{v>0\}$.

Let us suppose that this is not the case. Recall the sup and inf-convolutions \eqref{sub-inf-conv}. There exist $r > 0$ sufficiently small and $a > 1$ such that $U^r \subset a V_r$, $\partial U^r \cap a \partial V_r \neq \emptyset$ and $K^r \subset aK_r$ by convexity of $K$. Set $\tilde v(x) := v_r(a^{-1}x)$. By maximum principle for harmonic functions, $u^r \leq \tilde v$. Moreover $\tilde v$ is a strict supersolution.

Let us refer to $u^r$, $\tilde v$, $U^r$ and $a V_r$ as $u$, $v$, $U$, $V$, respectively, in the following.
    Let $x_0\in \partial U \cap \partial V$. Since $U$ is inner regular and $V$ outer regular, there is a unique supporting normal $n_0$ at $x_0$. 
    Let $P_0$ be the supporting hyperplane to $V$ at $x_0$.  Then $\Gamma := \partial \{u>0\} \cap P_0$ is compact and all points of $\Gamma$ are uniformly inner and outer regular points of $\partial V$ and, by \lref{supersoln-blowup},
\[|\grad v(x)| \leq Q(n_{0})-\delta \ \hbox{ for all } \ x \in \Gamma,\]
for some $\delta>0$ by the strict supersolution property.
By \cite[Lemma 3.14]{FeldmanSmart} 
\[v(x) \leq (Q(n_0)-\delta+Cr^{1/3})((x-x_0) \cdot n_0)_+ \ \hbox{ on } \ \Gamma + B_r(0)\]
so taking $r = c\delta^3$ we find that
\[\varphi(x) = (Q(n_0)-\tfrac{1}{2}\delta)((x-x_0) \cdot n_0)_+\]
touches $u$ from above on $\Gamma$ with strict ordering on $\partial (\Gamma + B_r(0))$.  This contradicts the weak subsolution property of $u$.
\end{proof}
  \subsection{Comparison principle: convex subsolution}
Next we consider the case when a supersolution touches a convex weak subsolution from above.  In \cite[Lemmas 3.15 and 3.16]{FeldmanSmart} this case relied on a geometric argument using convexity allowing to compare the gradient of the touching supersolution at a sequence of points with normal $\nu_n$ approaching the touching direction $\nu_0$ to the gradient of the subsolution on its $\nu_n$ facet.  In place of this argument we are able to take advantage of well-known results from convex analysis in the case of continuous $Q$.  The idea is that the weak subsolution condition is not weak at exposed points, and we can exploit the quasi-convexity of the gradient along with Straszewicz's Theorem, \tref{Straszewicz}, to derive a strong subsolution condition on exposed faces.
  \begin{lemma}
\label{l.subsol-grad-bound}
Suppose that $v$ is a weak subsolution of \eqref{e.minsup} with continuous $Q$, harmonic in $\{v>0\} \setminus K$. Further assume that $V = \{v>0\}$ is convex and uniformly inner regular, and $n$ is some normal direction with facet $V_n = \{x \in \overline{V}: n \cdot x = \inf_{x \in V} x \cdot n\}$.   Then
\[ \min_{V_n} |\grad v| \geq Q(n).\]
\end{lemma}
    Note that the weak subsolution property just says $\max_{V_n}|\grad v| \geq Q(n)$.  Actually this allows us to conclude, in the setting of the Lemma, that $V_n$ must be trivial (a single point).  If we did not assume that $Q(n)$ was continuous we would only obtain
 \[ \min_{V_n} |\grad v| \geq \liminf_{n' \to n}Q(n').\]

 \begin{proof}

Since $V_n$ is the convex hull of its extreme points by \lref{extremehull} and the superlevel set $\{x \in V_n : |\grad v|(x) \geq Q(n)\}$ is convex by \lref{grad-convex}, it is enough to establish the inequality at the extreme points of $V_n$.

Let $x_0$ be an extreme point of $V_n$. By inner regularity, $n$ is the inner unit normal of $\partial V$ at $x_0$. Clearly $x_0$ is also an extreme point of $V$. % \cite[Sec.~18]{Rockafellar}
By \tref{Straszewicz} there is a sequence $x_j \to x_0$ of exposed points of $V$. 
At exposed points of $V$, we have by the subsolution condition directly
\begin{align*}
|\grad v|(x_j) \geq Q(n_{x_j}).
\end{align*}
Since $\partial V$ is uniformly inner and outer regular 
\[ |\grad v|(x_0) = \lim_{n \to \infty} |\grad v|(x_j) \geq \lim_{j \to \infty}Q(n_{x_j}) = Q(n),\]
where we used the standard gradient regularity in $C^{1,1}$ domains, see \cite[Lemma 3.13]{FeldmanSmart} for proof.

  \end{proof}

Now applying \lref{subsol-grad-bound} along with a typical dilation argument to create a touching point we will get a comparison principle.
  \begin{lemma}\label{l.comparison-convex-below}
Assume $Q$ is continuous. Suppose that $u$ is supersolution of \eref{minsup}, and $v$ is a weak subsolution of \eref{minsup} with $\{v>0\}$ convex.  Then $v \leq u$.
\end{lemma}

\begin{proof}
By translation we can assume that the interior of $K$ contains the origin.
Since $v$ and $u$ are harmonic in their positive sets, it is sufficient to show that $V := \{v>0\} \subset \{u>0\} =: U$.

Let us suppose that this is not the case. Recall the sup and inf-convolutions \eqref{sub-inf-conv}. There exist $r > 0$ sufficiently small and $a > 1$ such that $V^r \subset a U_r$, $\partial V^r \cap a \partial U_r \neq \emptyset$ and $K^r \subset aK_r$ by convexity of $K$. Set $\tilde u(x) := u_r(a^{-1}x)$. By maximum principle for harmonic functions, $v^r \leq \tilde u$.

At any point $x \in \partial V^r \cap a \partial U_r$ we have $|\nabla v^r| \geq Q(n)$ by \lref{subsol-grad-bound}. However, this is a contradiction with the fact that $\tilde u$ cannot be touched from below by a test function with a slope larger than $Q(n)/ a$ by supersolution property.
\end{proof}

\subsection{Existence and uniqueness of a quasi-convex solution}
Now we can use the comparison principle, in the case of continuous $Q$, to show uniqueness and quasi-convexity of solutions.
\begin{theorem}[{cf. \cite[Th.~3.10]{FeldmanSmart}}]\label{t.minconvex}
Assume that $K$ is nonempty convex compact set with inner regular boundary and assume that $Q$ is continuous. Then there exists a unique continuous function $u$ with compact support that is both a supersolution and a weak subsolution. Moreover $\{u>0\}$ is convex. 
\end{theorem}

\begin{proof}
The existence follows by Perron's method taking the infimum $u$ of supersolutions with convex support. The support of $u$ must be convex since all supersolutions are bounded from below by a harmonic function with $1$ on $K$ with support equal to the intersection of the convex support of the supersolutions.
A standard argument yields that $u$ is also a supersolution. And $u$ must be a weak subsolution for otherwise we could touch it from above by a one dimensional $C^\infty$ test function that is a strong supersolution on a neighborhood of the contact set. By shifting this test function down a smaller supersolution with convex support can be created, leading to a contradiction. Finally, uniqueness follows from the convex comparison in \lref{comparison-convex-above} and \lref{comparison-convex-below}.
\end{proof}

The comparison principle and convexity in the case of upper semicontinuous $Q$ now follows by a monotone approximation argument. 

\begin{corollary}\label{c.cts-approx}
Suppose $Q$ is upper semicontinuous and let $u$ be the minimal supersolution of \eref{minsup}. Then:
\begin{enumerate}
    \item $\{u>0\}$ is convex
    \item Any supersolution is larger than $u$, and $u$ is larger than any weak subsolution, and therefore any supersolution which is also a weak subsolution is identical to $u$.
\end{enumerate}
\end{corollary}

  \begin{proof}
  Since $Q$ is upper semi-continuous there is a monotone decreasing sequence of continuous $Q^j : S^{d-1} \to (0,\infty)$ with $Q^j \searrow Q$. Let $u_j$ be the minimal supersolution of \eref{minsup} corresponding to the equation $|\grad u^j| \leq Q^j(n_x)$ on $\partial \{u^j>0\}$, and $u$ be the minimal supersolution corresponding to $|\grad u| \leq  Q(n)$.  
  
  Since $Q^j \geq Q$ we have $u$ is a supersolution of $|\grad u| \leq Q^j(n_x)$ on $\partial \{u>0\}$ so
  \[ u \geq u^j.\]
  Similarly, since the $Q^j$ are monotone decreasing, the $u^j$ are a monotone increasing sequence.  
  
  Call $u^\infty := \lim_{j \to \infty} u^j \leq u$.  We claim that $u^\infty$ is a supersolution of \eref{minsup} which will mean $u^\infty \geq u$ and hence $u^\infty = u$.   The supersolution condition is standard stability of the viscosity solution property with respect to uniform convergence, \lref{stability}.

  By \tref{minconvex} the sets $\{u^j>0\}$ are convex for all $j$.  Since they converge monotonically upwards to $\{u>0\}$ that set is convex as well. 
  
  Now let $v$ be a weak subsolution, since we now know that $\{u>0\}$ is convex we can apply \lref{comparison-convex-above} to find $v \leq u$. If $v$ is also a supersolution of \eref{minsup} then $v \geq u$, by the minimality of $u$, and so we conclude that $v = u$. 
  \end{proof}

\bibliography{comparison-articles.bib}

\begin{bibdiv}
\begin{biblist}

\bib{Acker}{article}{
      author={Acker, Andrew},
       title={A free boundary optimization problem},
        date={1978},
     journal={SIAM Journal on Mathematical Analysis},
      volume={9},
      number={6},
       pages={1179\ndash 1191},
}

\bib{Beurling}{book}{
      author={Beurling, Arne},
       title={On free boundary problems for the laplace equation},
   publisher={Institute for Advanced Study},
     address={Princeton},
        date={1957},
}

\bib{Bianchini2012}{article}{
      author={Bianchini, Chiara},
       title={A {B}ernoulli problem with non-constant gradient boundary
  constraint},
        date={2012},
        ISSN={0003-6811,1563-504X},
     journal={Appl. Anal.},
      volume={91},
      number={3},
       pages={517\ndash 527},
         url={https://doi.org/10.1080/00036811.2010.549479},
      review={\MR{2876741}},
}

\bib{Bianchini2022}{incollection}{
      author={Bianchini, Chiara},
       title={Geometric properties for a {F}insler {B}ernoulli exterior
  problem},
        date={2022},
   booktitle={Current trends in analysis, its applications and computation},
      series={Trends Math. Res. Perspect.},
   publisher={Birkh\"{a}user/Springer, Cham},
       pages={423\ndash 432},
         url={https://doi.org/10.1007/978-3-030-87502-2_43},
      review={\MR{4559538}},
}

\bib{CaffarelliLee}{article}{
      author={Caffarelli, L.},
      author={Lee, K.},
       title={Homogenization of oscillating free boundaries: the elliptic
  case},
        date={2007},
        ISSN={0360-5302,1532-4133},
     journal={Comm. Partial Differential Equations},
      volume={32},
      number={1-3},
       pages={149\ndash 162},
         url={https://doi.org/10.1080/03605300600635038},
      review={\MR{2304145}},
}

\bib{CaffarelliMellet}{article}{
      author={Caffarelli, L.~A.},
      author={Mellet, A.},
       title={Capillary drops: contact angle hysteresis and sticking drops},
        date={2007},
        ISSN={0944-2669,1432-0835},
     journal={Calc. Var. Partial Differential Equations},
      volume={29},
      number={2},
       pages={141\ndash 160},
         url={https://doi.org/10.1007/s00526-006-0036-y},
      review={\MR{2307770}},
}

\bib{CaffarelliSalsa}{book}{
      author={Caffarelli, Luis},
      author={Salsa, Sandro},
       title={A geometric approach to free boundary problems},
      series={Graduate Studies in Mathematics},
   publisher={American Mathematical Society, Providence, RI},
        date={2005},
      volume={68},
        ISBN={0-8218-3784-2},
         url={https://doi.org/10.1090/gsm/068},
      review={\MR{2145284}},
}

\bib{CaffarelliSpruck}{article}{
      author={Caffarelli, Luis~A.},
      author={Spruck, Joel},
       title={Convexity properties of solutions to some classical variational
  problems},
        date={1982},
        ISSN={0360-5302,1532-4133},
     journal={Comm. Partial Differential Equations},
      volume={7},
      number={11},
       pages={1337\ndash 1379},
         url={https://doi.org/10.1080/03605308208820254},
      review={\MR{678504}},
}

\bib{Feldman}{article}{
      author={Feldman, William~M.},
       title={Limit shapes of local minimizers for the alt-{C}affarelli energy
  functional in inhomogeneous media},
        date={2021},
        ISSN={0003-9527,1432-0673},
     journal={Arch. Ration. Mech. Anal.},
      volume={240},
      number={3},
       pages={1255\ndash 1322},
         url={https://doi.org/10.1007/s00205-021-01635-6},
      review={\MR{4264946}},
}

\bib{FeldmanSmart}{article}{
      author={Feldman, William~M.},
      author={Smart, Charles~K.},
       title={A free boundary problem with facets},
        date={2019},
        ISSN={0003-9527},
     journal={Arch. Ration. Mech. Anal.},
      volume={232},
      number={1},
       pages={389\ndash 435},
         url={https://doi.org/10.1007/s00205-018-1323-4},
      review={\MR{3916978}},
}

\bib{Hamilton}{article}{
      author={Hamilton, Richard~S.},
       title={The inverse function theorem of {N}ash and {M}oser},
        date={1982},
        ISSN={0273-0979,1088-9485},
     journal={Bull. Amer. Math. Soc. (N.S.)},
      volume={7},
      number={1},
       pages={65\ndash 222},
         url={https://doi.org/10.1090/S0273-0979-1982-15004-2},
      review={\MR{656198}},
}

\bib{HenrotShahgholian1997}{article}{
      author={Henrot, Antoine},
      author={Shahgholian, Henrik},
       title={Convexity of free boundaries with {B}ernoulli type boundary
  condition},
        date={1997},
        ISSN={0362-546X,1873-5215},
     journal={Nonlinear Anal.},
      volume={28},
      number={5},
       pages={815\ndash 823},
         url={https://doi.org/10.1016/0362-546X(95)00192-X},
      review={\MR{1422187}},
}

\bib{HenrotShahgholian2002}{article}{
      author={Henrot, Antoine},
      author={Shahgholian, Henrik},
       title={The one phase free boundary problem for the {$p$}-{L}aplacian
  with non-constant {B}ernoulli boundary condition},
        date={2002},
        ISSN={0002-9947,1088-6850},
     journal={Trans. Amer. Math. Soc.},
      volume={354},
      number={6},
       pages={2399\ndash 2416},
         url={https://doi.org/10.1090/S0002-9947-02-02892-1},
      review={\MR{1885658}},
}

\bib{Kim}{article}{
      author={Kim, Inwon~C.},
       title={Homogenization of a model problem on contact angle dynamics},
        date={2008},
        ISSN={0360-5302,1532-4133},
     journal={Comm. Partial Differential Equations},
      volume={33},
      number={7-9},
       pages={1235\ndash 1271},
         url={https://doi.org/10.1080/03605300701518273},
      review={\MR{2450158}},
}

\bib{Rockafellar}{book}{
      author={Rockafellar, R.~Tyrrell},
       title={Convex analysis},
      series={Princeton Mathematical Series, No. 28},
   publisher={Princeton University Press, Princeton, N.J.},
        date={1970},
      review={\MR{0274683}},
}

\bib{Schaeffer}{article}{
      author={Schaeffer, David~G.},
       title={A stability theorem for the obstacle problem},
        date={1975},
        ISSN={0001-8708},
     journal={Advances in Mathematics},
      volume={17},
      number={1},
       pages={34\ndash 47},
  url={https://www.sciencedirect.com/science/article/pii/0001870875900857},
}

\end{biblist}
\end{bibdiv}
\end{document}